\documentclass{amsart}
\usepackage[english]{babel}
\usepackage{amsthm}
\usepackage{inputenc}
\usepackage{amssymb}
\usepackage{graphicx}
\usepackage{color}
\newtheorem{thm}{Theorem}[section]

\newtheorem{defn}[thm]{Definition}

\numberwithin{equation}{section}

\newcommand{\Der}{\textrm{Der}}

\begin{document}

\title[Solvable Leibniz algebras with abelian nilradical]{Classification of solvable Leibniz algebras with abelian nilradical and $k-1$ dimensional extension}%
\author{Gaybullaev R.K., Khudoyberdiyev A.Kh., Pohl K.}

\address{[Khudoyberdiyev \ A.\ Kh.]
National University of Uzbekistan, Institute of Mathematics Academy of Sciences of Uzbekistan, Tashkent, 100174, Uzbekistan.}
\email{khabror@mail.ru}

\address{[Gaybullaev \ R. \ K]
National University of Uzbekistan,  100174, Uzbekistan.}
\email{r\_gaybullaev@mail.ru}

\address{[Pohl \ K.]
St. Olaf College,
1500 St. Olaf Avenue, Northfield, Minnesota 55057, USA.}
\email{pohl1@stolaf.edu}

%

\begin{abstract}
This work is devoted to the classification of solvable Leibniz algebras with an abelian nilradical. We consider $k-1$ dimensional extension of $k$-dimensional abelian algebras and classify all $2k-1$-dimensional solvable Leibniz algebras with an abelian nilradical of dimension $k$.

\end{abstract}

\maketitle


\maketitle

\section{Introduction.}

Leibniz algebras, a ``noncommutative version'' of Lie algebras,
were first introduced in the mid- 1960's by Blokh \cite{Bl} under
the name ``$D$-algebras". They appeared again in the 1990's after
Loday's work \cite{Loday1}, where he reintroduced them, coining them
``Leibniz algebras".

According to the structural theory of Lie algebras, a finite-dimensional Lie algebra can be written as a semidirect sum of its semisimple subalgebra and its
solvable radical (Levi's theorem). The semisimple part is a direct sum of simple Lie algebras which were completely classified in the fifties of the last century.
In the case of Leibniz algebras,
there is also an analogue to Levi's theorem \cite{Bar}. Namely, the
decomposition of a Leibniz algebra into the semidirect sum of its
solvable radical and a semisimple Lie algebra can be obtained. The
semisimple part can be described from simple Lie ideals (see
\cite{Jac}). Therefore, today there is a focus on studying the
solvable radical.

Owing to a result of \cite{Muba}, an approach to the study of
solvable Lie algebras through the use of the nilradical was developed in
\cite{Ancochea, NdWi,SnWi}, etc. In particular, in \cite{NdWi} solvable Lie algebras with abelian nilradicals are investigated.

The analogue of
Mubarakzjanov's \cite{Muba} result has been applied in the Leibniz
algebra case in \cite{Casas}, showing the importance of
consideration of the nilradical in the case of Leibniz algebras as
well. The papers \cite{CanKhud,Casas,Lad} also are devoted
to the study of solvable Leibniz algebras by considering the
nilradical.

It should be noted that any solvable Leibniz algebra $L$ with nilradical $N$ can be written as a direct sum of vector spaces $L =N \oplus Q,$ where $Q$ is the complementary vector space to the nilradical. In \cite{Ad} and \cite{LinBos}, solvable Leibniz algebras with an abelian nilradical are investigated. It was proven that the maximal dimension of a solvable Leibniz algebra with a $k$-dimensional abelian nilradical is $2k.$ Additionally, in  \cite{Ad} this maximal case was classified and some results regarding the classification with a $1$-dimensional extension were presented. In this paper we give the classification of solvable Leibniz algebras with abelian nilradical and $k-1$ dimensional extension.

It should be noted that a solvable Leibniz algebra $L$ with condition $dimQ = dim(N/N^2)$ can be classified using the classification of solvable Leibniz algebras with a $k$-dimensional abelian nilradical and a $k$-dimensional complementary vector space.

The natural next step is the classification of solvable Leibniz algebra with condition $dimQ = dim(N/N^2)-1.$ In order to perform this classification, the classification of solvable Leibniz algebras with a $k$-dimensional abelian nilradical and a $k-1$-dimensional complementary vector space should first be obtained. In the case $k=2$ we have three dimensional algebras which were classified in \cite{Casas 3-dim}, and the case $k=3$ is also already given in \cite{Khud}.
 In this paper, we classify the $2k-1$ case for any $k$.

Throughout this paper all algebras (vector spaces) considered are finite-dimensional and over the field of complex numbers. Also, in the tables of multiplications of algebras, we give nontrivial products only.

\section{Preliminaries}

This section is devoted to recalling some basic notions and concepts used throughout the paper.

\begin{defn}
A vector space with a bilinear bracket $(L,[\cdot,\cdot])$
is called a Leibniz algebra if for any $x,y,z\in L$ the so-called
Leibniz identity
$$\big[x,[y,z]\big]=\big[[x,y],z\big]-\big[[x,z],y\big],$$
holds.
\end{defn}

Here, we adopt the right Leibniz identity; since the bracket is
not skew-symmetric, there exists the version corresponding to the
left Leibniz identity,
 \[\big[[x,y],z\big] =  \big[x,[y,z]\big] - \big[y,[x,z]\big] \,. \]

The sets $Ann_r(L)=\{x \in L: [y,x]=0, \ \forall y \in L\}$ and
$Ann_l(L)=\{x \in L: [x,y]=0, \ \forall y \in L\}$
 are called the right and left annihilators of $L$, respectively.
It is observed that $Ann_r(L)$ is a
two-sided ideal of $L$, and for any $x, y \in L$ the elements $[x,x]$ and $[x,y]+
[y,x]$ are always in $Ann_r(L)$.

The set $C(L)=\{z \in L: [x,z]=[z,x]=0, \ \forall x \in L\}$
is called the Center of $L$.

For a given Leibniz algebra $(L,[\cdot,\cdot])$ the sequences of
two-sided ideals is defined recursively as follows:
\[L^1=L, \ L^{k+1}=[L^k,L],  \ k \geq 1, \qquad \qquad
L^{[1]}=L, \ L^{[s+1]}=[L^{[s]},L^{[s]}], \ s \geq 1.
\]
These are said to be the lower central and the derived series of $L$,
respectively.

\begin{defn} A Leibniz algebra $L$ is said to be
nilpotent (respectively, solvable), if there exists $n\in\mathbb
N$ ($m\in\mathbb N$) such that $L^{n}=0$ (respectively,
$L^{[m]}=0$).
\end{defn}

\begin{defn}
  An ideal of a Leibniz algebra is called nilpotent if it is nilpotent as a subalgebra.
  \end{defn}
It is easy to see that the sum of any two nilpotent ideals is nilpotent. Therefore the maximal nilpotent ideal always exists.

 \begin{defn}
  The maximal nilpotent ideal of a Leibniz algebra is said to be the nilradical of the algebra.
  \end{defn}

\begin{defn} A linear map $d \colon L \rightarrow L$ of a Leibniz algebra $(L,[\cdot,\cdot])$ is said to be a
 derivation if for all $x, y \in L$, the following  condition holds: \[d([x,y])=[d(x),y] + [x, d(y)] \,.\]
\end{defn}
The set of all derivations of $L$ is denoted by $\Der(L).$ The set $\Der(L)$ is a Lie algebra with respect to the commutator.

For a given element $x$ of a Leibniz algebra $L$, the right
multiplication operator $R_x \colon L \to L,$ defined by $R_x(y)=[y,x],  y \in
L$ is a derivation. In fact, Leibniz algebras are characterized by this property regarding right multiplication operators. (Recall that left Leibniz algebras are characterized by the same property with left multiplication operators.) As in the Lie case, these kinds of derivations are said to be inner derivations.



\begin{defn}
Let $d_1, d_2, \dots, d_n$ be derivations of a Leibniz algebra
$L.$ The derivations $d_1, d_2, \dots, d_n$ are said to be linearly
nil-independent if for $\alpha_1, \alpha_2, \dots ,\alpha_n \in \mathbb{C}$
and a natural number $k$,
$$(\alpha_1d_1 + \alpha_2d_2+ \dots +\alpha_nd_n)^k=0\ \hbox{implies }\ \alpha_1 = \alpha_2= \dots =\alpha_n=0.$$
\end{defn}

Note that in the above definition the power is understood with respect to composition.

Let $L$ be a solvable Leibniz algebra. Then it can be written
in the form $L=N \oplus Q,$ where $N$ is the nilradical and $Q$ is the
complementary subspace. The following is a result from  \cite{Casas} on the dimension of $Q$ which we make use of in the paper.

\begin{thm}\label{nilr}
Let $L$ be a solvable Leibniz algebra and $N$ be its nilradical. Then
the dimension of $Q$ is not
greater than the maximal number of nil-independent derivations of
$N.$
\end{thm}

\section{Main result}

We denote by $\mathbf{a_k}$ a $k$-dimensional abelian algebra and by $R(\mathbf{a_k}, s)$ the class of solvable Leibniz algebras with a $k$ dimensional nilradical $a_k$ and an $s$-dimensional complementary vector space.

As above, it has been proven that $s\leq k$ for any algebra from the class $R(\mathbf{a_k}, s)$, and in \cite{Ad} the classification of such algebras of $R(\mathbf{a_k}, k)$  is given.
It is proven that an arbitrary algebra of the family $R(\mathbf{a_k}, k )$ can be decomposed into a direct sum of copies of two-dimensional non-trivial solvable Leibniz algebras.

Thus, any algebra $L$ from the family $R(\mathbf{a_k}, k )$ is
 $$L = l_2 \oplus l_2 \oplus \dots \oplus l_2 \oplus r_2 \oplus r_2 \oplus \dots  \oplus r_2,$$
where $l_2 : [e, x] = e$ and $r_2 : [e, x] = e, [x, e] = -e .$

Let $L$ be a Leibniz algebra from the class $R(\mathbf{a_k}, k-1).$ Take a basis
$\{e_1, e_2, \dots, e_k, x_1, x_2,\dots, x_{k-1}\}$ of $L$ such that $e_1, e_2, \dots, e_k$ is a basis of nilradical $N$ and $x_1, x_2,\dots, x_{k-1}$ is a basis of the complementary vector space $Q$. It is known that  the right multiplication operators, $R_{x_1}, R_{x_2}, \dots, R_{x_{k-1}}:N\to N$ are nil-independent derivations and there exists a basis  $\{e_1, e_2, \dots, e_k\}$ of $N$, such that $R_{x_1}, R_{x_2}, \dots, R_{x_{k-1}}$ simultaneously all have Jordan forms.

Because of this, observe that:
$$ \begin{cases}[e_i,x_j]=a_{i,j}e_i+\beta_{i,j}e_{i+1}, & 1 \leq i,j \leq k-1,\\
[e_k,x_j]=a_{k,j}e_k & 1 \leq j \leq k-1,\end{cases}$$
where $a_{i,j}$ are eigenvalues of the operator $R_{x_j}$ and $\beta_{i,j} \in \{0,1\}.$

Since $R_{x_1}, R_{x_2}, \dots, R_{x_{k-1}}$ are nil-independent we have that
$$rank \left(\begin{array}{cccc} a_{1,1}&a_{1,2}&\dots&a_{1,k-1}\\ a_{2,1}&a_{2,2}&\dots&a_{2,k-1}\\
\vdots&\vdots&\vdots&\vdots \\
a_{k-1,1}&a_{k-1,2}&\dots&a_{k-1,k-1}\\
a_{k,1}&a_{k,2}&\dots&a_{k,k-1}\end{array}\right) = k-1.$$

Thus, there exists a minor $M_{1,2,\dots, k-1}^{i_1,i_2\dots, i_{k-1}}$ of order $k-1$ which has a non-zero determinant, i.e. there exists a $t$ such that  $Det(M_{1,2,\dots, k-1}^{1,\dots, t-1, t+1, \dots, k})\neq 0.$
Making the change of basis $$\begin{cases} e'_i=e_i, & 1 \leq i \leq t-1, \\
e_i' = e_{i+1}& t \leq i \leq k-1,\\
e_k'=e_t\end{cases}$$
we get that
$$ \begin{cases}[e_i,x_j]=a_{i,j}e_i+\beta_{i,j}e_{i+1}, & 1 \leq i,j \leq k-1, \ i \neq t-1,\\
[e_{t-1},x_j]=a_{t-1,j}e_{t-1}+\beta_{t-1,j}e_{k}, & 1 \leq j \leq k-1,\\
[e_k,x_j]=a_{k,j}e_k+\beta_{k,j}e_t & 1 \leq j \leq k-1.\end{cases}$$

It should be noted that operators $R_{x_1}, R_{x_2}, \dots, R_{x_{k-1}}$ can be considered linearly nil-independent operators on the quotient vector space $\mathbf{a_k}/\langle e_k\rangle.$ Since $dim(\mathbf{a_k}/\langle e_k\rangle)=k-1$ from the result of \cite{Ad} we obtain that

$$a_{i,i}=1, \quad 1 \leq i \leq k-1, \qquad
a_{i,j}=0, \quad  1 \leq i, j (i\neq j)\leq k-1,$$
$$\beta_{i,j} =0, \quad 1\leq i \ (i\neq t-1)\leq k, \quad 1\leq j\leq k-1.$$

Let us introduce the following notation:
$$[x_i,e_j] = \sum\limits_{p=1}^k\gamma_{i,j}^pe_p, \quad 1 \leq i \leq k-1, \ 1 \leq j \leq k,$$
$$[x_i,x_j] = \sum\limits_{p=1}^k\delta_{i,j}^pe_p, \quad 1 \leq i,j \leq k-1.$$

Using the similar algorithms of the proof of Theorem 3.2 in \cite{Ad} from Leibniz identities and basis changes we obtain that

$$\gamma_{i,j}^p =0, \quad 1 \leq i,j,p \leq k-1, \quad i \neq j \neq p,$$
$$\gamma_{i,i}^i \in \{0, -1\}, \quad 1 \leq i \leq k-1,$$
$$\delta_{i,j}^p=0, \quad  1 \leq i,j,p \leq k-1, $$

Therefore, the multiplication of the $2k-1$-dimensional solvable Leibniz algebra with $k$-dimensional abelian nilradical has the following form:
\begin{equation}\label{eq-main}R(\mathbf{a_k}, k-1) = \left\{\begin{array}{llll}
  [e_i,x_i]=e_i+\beta_{ii}e_k, &  1\le i\leq k-1,\\[1mm]
  [e_i,x_j]= \beta_{ij}e_k, &  1\le i\leq k,&1\le j\leq k-1,& i\not=j,\\[1mm]
  [x_i,e_i]=\alpha_ie_i+\gamma_{i,i}e_k, &  1\le i\leq k-1,\\[1mm]
  [x_i,e_j]=\gamma_{i,j}e_k ,&  1\le i\leq k-1,& 1\le j\leq k-1, & i\not=j,\\[1mm]
   [x_i,e_k]=\sum\limits_{j=1}^k\nu_{i,j}e_j,&  1\le i\leq k-1,\\[1mm]
  [x_i,x_j]=\delta_{i,j}e_k, &  1\le i,j\leq k-1,
\end{array}\right.
\end{equation}
where $\alpha_i \in \{0,\;-1\}.$

First we investigate the case of $\alpha_i=0$ for $1\leq i \leq k-1.$

\begin{thm}\label{thm3.1} Let $L$ be a Leibniz algebra from the class $R(\mathbf{a_k}, k-1)$ and let $\alpha_i=0$ for $1\leq i \leq k-1.$ Then $L$ is isomorphic to one of the following algebras:
$$\begin{array}{ll} L_1(\beta_{i}):\begin{cases}
  [e_i,x_i]=e_i, &  1\le i\leq k-1,\\
  [e_k,x_i]= \beta_{i}e_k, &  1\le i\leq k-1,
\end{cases} & L_2(\beta_{i}):\begin{cases}
  [e_i,x_i]=e_i, &  1\le i\leq k-1,\\
  [e_k,x_i]= \beta_{i}e_k, &  1\le i\leq k-1,\\
    [x_i,e_k]=-\beta_{i}e_k, &  1\le i\leq k-1,
\end{cases} \\ L_3(\beta_i):\begin{cases}
  [e_1,x_1]=e_1+\beta_1e_k, \\
  [e_i,x_i]=e_i, &  2\le i\leq k-1,\\
  [e_1,x_i]= \beta_ie_k, &  2\le i\leq k-1,\\
  [e_k,x_1]=e_k,
\end{cases} &
  L_4(\nu_{i}):\begin{cases}
  [e_i,x_i]=e_i, &  1\le i\leq k-1,\\
  [e_k,x_1]=e_k,\\
  [x_1,e_k]=-e_k,\\
  [x_i,e_k]=\nu_{i}e_1, &  2\leq i\leq k-1,
\end{cases} \\ L_5(\delta_{i,j}):\begin{cases}
  [e_i,x_i]=e_i, &  1\le i\leq k-1,\\
  [x_i,x_j]=\delta_{i,j}e_k, &  1\le i,j\leq k-1.
\end{cases}\end{array}$$
\end{thm}

\begin{proof}
Let $\alpha_i=0$ for $1\leq i \leq k-1,$ then the multiplication \eqref{eq-main} has the form

$$\left\{\begin{array}{llll}
  [e_i,x_i]=e_i+\beta_{i,i}e_k, &  1\le i\leq k-1,\\[1mm]
  [e_i,x_j]= \beta_{i,j}e_k, &  1\le i\leq k, & 1\le j\leq k-1,& i\neq j,\\[1mm]
  [x_i,e_j]=\gamma_{i,j}e_k, &  1\le i\leq k-1,& 1\le j\leq k-1,\\[1mm]
  [x_i,e_k]=\sum\limits_{j=1}^k\nu_{i,j}e_j,&  1\le i\leq k-1,\\[1mm]
  [x_i,x_j]=\delta_{i,j}e_k, &  1\le i,j\leq k-1.
\end{array}\right.$$

\textbf{Case 1.}
Let there exist a $\beta_{k,i}\not=\{0,1\}$. Without loss of generality, we may assume $\beta_{k,1}\not=0$. Changing the basis, let
$$e_1'=e_1-\frac{\beta_{1,1}}{\beta_{k,1}-1}e_k, \qquad e_i'=e_i-\frac{\beta_{i,1}}{\beta_{k,1}}e_k, \quad 2\leq i\leq k-1.$$ Then

$$[e_1',x_1]=[e_1-\frac{\beta_{1,1}}{\beta_{k,1}-1}e_k,x_1]=e_1+\beta_{1,1}e_k-\frac{\beta_{1,1}}{\beta_{k,1}-1}\beta_{k,1}e_k=e_1-\frac{\beta_{1,1}}{\beta_{k,1}-1}e_k=e_1'.$$
$$[e_i',x_1]=[e_i-\frac{\beta_{i,1}}{\beta_{k,1}}e_k,x_1]=\beta_{i,1}e_k-\beta_{i,1}e_k=0,  \quad 2\leq i\leq k-1.$$

So $e_1'x_1=e_1'$ and $e_i'x_1=0$ for $2\leq i\leq k-1$.
Then using the Leibniz identity, we have:
$$0=[e_i,[x_1,x_j]]=[[e_i,x_1]x_j]-[[e_i,x_j],x_1]=\beta_{i,j}[e_k,x_1]=-\beta_{i,j}\beta_{k,1}e_k,$$
$$0=[e_j,[x_1,x_j]]=[[e_j,x_1],x_j]-[[e_j,x_j],x_1]=[e_j+\beta_{j,j}e_k,x_1]=\beta_{j,j}\beta_{k,1}e_k.$$

Hence $\beta_{i,j}\beta_{k,1}=0,$ $\beta_{j,j}\beta_{k,1}=0$. Since $\beta_{k,1}$ is not zero, we know that $$\beta_{i,j}=0, \quad 2\leq i\leq k-1, \quad  2\leq j\leq k-1.$$

Next we consider
$$0=[e_1,[x_1,x_j]=[[e_1,x_1]x_j]-[[e_1,x_j]x_1]=[e_1,x_j]-[\beta_{1,j}e_k,x_1]=\beta_{1,j}e_k-\beta_{1,j}\beta_{k,1}e_k=\beta_{1,j}(1-\beta_{k,1})e_k.$$

Because $\beta_{k,1}\not=1$ we get $\beta_{1,j}=0$ for $2\leq j\leq k-1$.
Thus, our new multiplication is:
$$\left\{\begin{array}{lll}
  [e_i,x_i]=e_i, &  1\le i\leq k-1,\\[1mm]
  [e_k,x_i]= \beta_{k,i}e_k, &  1\le i\leq k-1,\\[1mm]
  [x_i,e_j]=\gamma_{i,j}e_k, &  1\le i\leq k-1,& 1\le j\leq k-1,\\[1mm]
  [x_i,e_k]=\sum\limits_{j=1}^k\nu_{i,j}e_j, &  1\le i\leq k-1,\\[1mm]
  [x_i,x_j]=\delta_{i,j}e_k, &  1\le i,j\leq k-1.
\end{array}\right.$$

Now we consider the Leibniz identity for the triple of elements $\{x_i, e_j, x_1\}$ for $1\leq i\leq k-1$ and $1\leq j\leq k-1.$ Then
$$0=[x_i,[e_1,x_1]]-[[x_i,e_1],x_1]+[x_i,x_1],e_i]=[x_i,e_1]-\gamma_{i,1}[e_k,x_1]=\gamma_{i,1}(1-\beta_{k,1})e_k,$$
$$0=[x_i,[e_j,x_1]-[[x_i,e_j]x_1]+[[x_i,x_1],e_j]=-\gamma_{i,j}[e_k,x_1]=-\gamma_{i,j}\beta_{k,1}e_k, \quad 2\leq j\leq k-1.$$

Since $\beta_{k,1}\neq \{0, 1\},$ we have $\gamma_{i,j}=0$ for $1\leq i,j\leq k-1$.

Using the Leibniz identity, we have:
$$0=[x_i,[e_k,x_1]-[[x_i,e_k],x_1]+[[x_i,x_1],e_k]=\beta_{k,1}[x_i,e_k]-\big[\sum\limits_{j=1}^k\nu_{i,j}e_j,x_1\big]$$
$$=\beta_{k,1}\sum\limits_{j=1}^k\nu_{i,j}e_j-\nu_{i,1}e_1-\nu_{i,k}\beta_{k,1}e_k=(\beta_{k,1}-1)\nu_{i,1}e_1+\beta_{k,1}\nu_{i,2}e_2+\dots+\beta_{k,1}\nu_{i,k-1}e_{k-1}.$$

Hence $\nu_{i,j}=0$ for $1\leq i, j \leq k-1$.

From
$$0=[x_i,[x_j,e_k]]-[[x_i,x_j]e_k]+[[x_i,e_k],x_j]=\nu_{j,k}[x_i,e_k]+\nu_{i,k}[e_k,x_j]=\nu_{i,k}(\nu_{j,k}+\beta_{k,j})e_k,$$
 we get

  \begin{equation}   \label{eq3.1} \nu_{i,k} (\nu_{j,k}+\beta_{k,j})=0, \quad 1\leq i,j\leq k-1. \end{equation}

Taking the basis change $x_i'=x_i-\frac{\delta_{i,1}}{\beta_{k,1}}e_k$ for $2\leq i\leq k-1,$ we obtain $$[x_i',x_1]=[x_i-\frac{\delta_{i,1}}{\beta_{k,1}}e_k,x_1]=\delta_{i,1}e_k-\delta_{i,1}e_k=0, \quad 2\leq i\leq k-1.$$

Thus, we can assume $\delta_{i,1}=0$ for $2\leq i\leq k-1$.

Using the Leibniz identity
$$0=[x_i,[x_1,x_1]]-[[x_i,x_1],x_1]+[[x_i,x_1],x_1]=\delta_{1,1}[x_i,e_k]=\nu_{i,k}\delta_{1,1}e_k$$
we have \begin{equation}\label{eq3.2} \nu_{i,k}\delta_{1,1}=0, \quad 1\leq i\leq k-1.\end{equation}

Another application of the Leibniz identity gives:
$$0=[x_i,[x_j,x_1]]-[[x_i,x_j],x_1]+[[x_i,x_1],x_j]=-\delta_{i,j}\beta_{k,1}e_k,$$
which implies $\delta_{i,j}=0$ for $2\leq i,j\leq k-1$.

From
$$0=[x_1,[x_j,x_i]]-[[x_1,x_j],x_i]+[[x_1,x_1],x_j]=-\delta_{1,j}[e_k,x_1]+\delta_{1,1}[e_k,x_j]=(-\delta_{1,j}\beta_{k,1}+\delta_{1,1}\beta_{k,j})e_k,$$
we have $\delta_{1,j}=\frac{\beta_{k,j}}{\beta_{k,1}}\delta_{1,1}$ for $2\leq j\leq k-1$.

Therefore, we have the following table multiplications:
$$\left\{\begin{array}{lll}
  [e_i,x_i]=e_i, &  1\le i\leq k-1,\\[1mm]
  [e_k,x_i]= \beta_{k,i}e_k, &  1\le i\leq k-1,\\[1mm]
  [x_i,e_k]=\nu_{i,k}e_k, &  1\le i\leq k-1,\\[1mm]
  [x_1,x_j]=\frac{\beta_{k,j}}{\beta_{k,1}}\delta_{1,1}e_k, &  1\le j\leq k-1.
\end{array}\right.$$

Let $\nu_{i,k}=0$ for all $i\;(1 \le i \leq k-1).$
Then taking the change $x_1'=x_1-\frac{\delta_{1,1}}{\beta_{k,1}}e_k$ we have:
$$[x_1',x_1']=[x_1-\frac{\delta_{1,1}}{\beta_{k,1}}e_k, x_1-\frac{\delta_{1,1}}{\beta_{k,1}}e_k]=\delta_{i,i}e_k-\delta_{i,i}e_k=0$$
$$[x_1',x_j]=[x_1-\frac{\delta_{1,1}}{\beta_{k,1}}e_k,x_j]=\frac{\beta_{k,j}}{\beta_{k,1}}\delta_{1,1}e_k-\frac{\delta_{1,1}}{\beta_{k,1}}\beta_{k,j}e_k=0.$$

Thus, we obtain the algebra $L_1(\beta_{i}),$ with $\beta_{1}\neq 0,1.$

Let there exist $i\;(1 \le i \leq k-1)$ such that $\nu_{i,k}\neq0$. According to the equalities \eqref{eq3.1} and \eqref{eq3.2} we have  $\nu_{i,k}=-\beta_{k,i}$ and $\delta_{1,1}=0$, which implies that $\delta_{1,j}=0$ for $2\leq j\leq k-1$.
Thus we have the algebra $L_2(\beta_{i}),$ with $\beta_{1}\neq 0,1.$

\textbf{Case 2.}
Let $\beta_{k,1}, \beta_{k,2},\dots,\beta_{k,k-1} \in \{0,1\}$. Assume $(\beta_{k,1},\beta_{k,2},...,\beta_{k,k-1})\neq(0,0,...,0).$ Without loss of generality, rearrange the basis elements such that the non-zero $\beta_{k,i}$ are the first $s$, where $1\leq s\leq k-1$. So we have
$\beta_{k,1}=\beta_{k,2}=\dots=\beta_{k,s}=1$ and $\beta_{k,s+1}=\beta_{k,s+2}=\dots=\beta_{k,k-1}=0$.

In this case the table of multiplications is
$$\left\{\begin{array}{llll}
  [e_i,x_i]=e_i+\beta_{i,i}e_k, &  1\le i\leq k-1,\\[1mm]
  [e_i,x_j]= \beta_{i,j}e_k, &  1\le i\leq k-1,& 1\le j\leq k-1,&  i\not=j,\\[1mm]
  [e_k,x_j]=e_k, &  1\le i\leq s,\\[1mm]
  [x_i,e_j]=\gamma_{i,j}e_k, &  1\le i\leq k-1,& 1\le j\leq k-1,\\[1mm]
  [x_i,e_k]=\sum\limits_{j=1}^k\nu_{i,j}e_j , &  1\le i\leq k-1,\\[1mm]
  [x_i,x_j]=\delta_{i,j}e_k, &  1\le i,j\leq k-1.
\end{array}\right.$$

Changing the basis, let $e_i'=e_i-\beta_{i,1}e_k$
for $2\leq i\leq k-1$ we have
$$[e_i',x_1]=[e_i-\beta_{i,1}e_k,x_1]=\beta_{i,1}e_k-\beta_{i,1}e_k=0, \quad 2\leq i\leq k-1.$$

Using the Leibniz identities
$$0=[e_i,[x_i,x_1]]-[[e_i,x_i]x_1]+[[e_i,x_1],x_i]=-[e_i+\beta_{i,i}e_k,x_1]+\beta_{i,1}[e_k,x_i]=-\beta_{i,i}e_k,$$
$$0=[e_i,[x_j,x_1]]-[[e_i,x_j],x_1]+[[e_i,x_1],x_j]=-\beta_{i,j}e_k,$$
we have $\beta_{i,j}=0$ for $2\leq i,j\leq k-1$.

Next, from
$$0=[x_i,[e_i,x_1]]-[[x_i,e_i],x_1]+[[x_i,x_1],e_i]=-\gamma_{i,i}[e_k,x_1]=-\gamma_{i,i}e_k,$$
$$0=[x_i,[e_j,x_1]]-[[x_i,e_j],x_1]+[[x_i,x_1],e_j]=-\gamma_{i,j}[e_k,x_1]=-\gamma_{i,j}e_k,$$
we obtain $\gamma_{i,j}=0$ for $1\leq i\leq k-1$, $2\leq j\leq k-1$.

First, consider:
$$0=[x_1,[e_k,x_1]]-[[x_1,e_k],x_1]+[[x_1,x_1],e_k]=[x_1,e_k]-\big[\sum\limits_{j=1}^k\nu_{1,j}e_j,x_1\big]=$$
$$=\sum\limits_{j=1}^k\nu_{1,j}e_j -\nu_{1,1}(e_1+ \beta_{1,1}e_k)- \nu_{1,k}e_k  = \nu_{1,2}e_2+\dots+\nu_{1,k-1}e_{k-1} -\nu_{1,1}\beta_{1,1}e_k.$$

Hence $$\nu_{1,1}\beta_{1,1}=0, \qquad \nu_{1,i}=0,\quad 2\le i \le k-1.$$

Now consider the Leibniz identity for the triples $\{x_i, e_k, x_i\}$ and $\{x_i, e_k, x_1\}.$

For $2\leq i\leq s,$ we have
$$0=[x_i,[e_k,x_i]]-[[x_i,e_k],x_i]+[[x_i,x_i],e_k]=\sum\limits_{j=1}^k\nu_{i,j}e_j-\big[\sum\limits_{j=1}^k\nu_{i,j}e_j,x_i\big]$$
$$=\sum\limits_{j=1}^k\nu_{i,j}e_j-\nu_{i,1}\beta_{1,i}e_k-\nu_{i,i}e_i-\nu_{i,k}e_k=\sum\limits_{j=1, j \neq i}^{k-1}\nu_{i,j}e_j-\nu_{i,1}\beta_{1,i}e_k,$$
which implies
$$\nu_{i,1} =  \dots = \nu_{i,i-1} = \nu_{i,i+1}=  \dots = \nu_{i,k-1} =0.$$

Next from $$0=[x_i,[e_k,x_1]]-[[x_i,e_k],x_1]+[[x_i,x_1],e_k]=\beta_{k,1}[x_i,e_k]-[\nu_{i,i}e_i+\nu_{i,k}e_k,x_1]=\nu_{i,i}e_i,$$
we get $\nu_{i,i}=0$ for $2\leq i\leq s$.
Thus, we obtain $x_ie_k=\nu_{i,k}e_k$ for $2\leq i\leq s$.

For $s+1\leq i\leq k-1$, we have
$$0=[x_i,[e_k,x_i]]-[[x_i,e_k],x_i]+[x_i,x_i],e_k]=-\big[\sum\limits_{j=1}^k\nu_{i,j}e_j,x_i\big]=-\nu_{i,1}\beta_{1,i}e_k-\nu_{i,i}e_i,$$
$$0=[x_i,[e_k,x_1]]-[[x_i,e_k],x_1]+[[x_i,x_1],e_k]=[x_i,e_k]-\big[\sum\limits_{j=1}^k\nu_{i,j}e_j,x_1\big]= \sum\limits_{j=2}^{k-1}\nu_{i,j}e_j-\nu_{i,1}\beta_{1,1}e_k.$$
which implies  $$\nu_{i,1}\beta_{1,i}=0 \quad \nu_{i,1}\beta_{1,1}=0, \quad \nu_{i,j}=0, \quad s<i\leq k-1, \quad 2\leq j\leq k-1.$$

Hence $[x_i,e_k]=\nu_{i,1}e_1+\nu_{i,k}e_k$ for $s+1\leq i\leq k-1$.

With the basis change $x_i'=x_i-\delta_{i,1}e_k$ for $2 \leq i \leq k-1$ we have:
$$[x_i',x_1]=[x_i-\delta_{i,1}e_k,x_1]=0.$$

Then, from
$$0=[x_i,[x_j,x_1]]-[[x_i,x_j],x_1]+[[x_i,x_1],x_j]=\delta_{i,j}[e_k,x_1]=-\delta_{i,j}e_k,$$
we obtain $\delta_{i,j}=0$ for $2\leq i,j\leq k-1$.

Therefore, the table of multiplications is:
$$\begin{cases}
  [e_1,x_1]=e_1+\beta_{1,1}e_k, \\
  [e_i,x_i]=e_i, &  2\le i\leq k-1,\\
  [e_1,x_j]= \beta_{1,j}e_k, &  2\le j\leq k-1,\\
  [e_k,x_j]=e_k, &  1\le j\leq s,\\
  [x_i,e_1]=\gamma_{i,1}e_k, &  1\le i\leq k-1,\\
  [x_1,e_k]=\nu_{1,1}e_1+\nu_{1,k}e_k,\\
  [x_i,e_k]=\nu_{i,k}e_k, &  2\le i\leq s,\\
  [x_i,e_k]=\nu_{i,1}e_1+\nu_{i,k}e_k, &  s+1\leq i\leq k-1,\\
  [x_1,x_j]=\delta_{1,j}e_k, &  1\le j\leq k-1,
\end{cases}$$
where
$\nu_{1,1}\beta_{1,1}=0,$
$\nu_{i,1}\beta_{1,1}=0,$ for $s<i\leq k-1.$

\textbf{Case 2.1.} Let $s\geq2,$ then using the Leibniz identities
$$0=[e_1,[x_1,x_2]]-[[e_1,x_1],x_2]+[[e_1,x_2],x_1]=-[e_1+\beta_{1,1}e_k,x_2]+\beta_{1,2}[e_k,x_1]=\beta_{1,1}e_k,$$
$$0=[e_1,[x_i,x_2]]-[[e_1,x_i],x_2]+[[e_1,x_2],x_i]=-\beta_{1,i}[e_k,x_2]+\beta_{1,2}[e_k,x_i]=(-\beta_{1,i}+\beta_{1,2})e_k, \quad 2\leq i\leq s,$$
$$0=[e_1,[x_j,x_2]-[[e_1,x_j],x_2]+[[e_1,x_2],x_j]=-\beta_{1,j}[e_k,x_2]+\beta_{1,2}[e_k,x_j]=-\beta_{1,j}e_k, \quad s+1\leq j\leq k-1,$$
$$0=[x_1,[e_k,x_2]-[[x_1,e_k],x_2]+[[x_1,x_2],e_k]=[x_1,e_k]-[\nu_{1,1}e_1+\nu_{1,k}e_k,x_2]= \nu_{1,1}e_1,$$
$$0=[x_i,[e_k,x_2]]-[[x_i,e_k],x_2]+[[x_i,x_2],e_k]=[x_i,e_k]-[\nu_{i,1}e_1+\nu_{i,k}e_k,x_2]=\nu_{i,1}e_1,\quad s+1 \leq i \leq k-1,$$
$$0=[x_i,[e_1,x_2]]-[[x_i,e_1],x_2]+[[x_i,x_2],e_1]=\gamma_{i,1}[e_k,x_2]=\gamma_{i,1}e_k, \quad 1 \leq i \leq k-1,$$
we obtain

$$\begin{array}{ccccc}\beta_{1,1}=0, & \beta_{1,i}=\beta_{1,2}, & 2\leq i\leq s, & \beta_{1,j}=0, & s+1\leq j\leq k-1,\\[1mm]
\nu_{1,1}=0, & \nu_{i,1}=0,& s+1 \leq i \leq k-1, & \gamma_{i,1}=0, & 1 \leq i \leq k-1.\end{array}$$

Making the basis change $e_1'=e_1-\beta_{1,2}e_k$ we obtain $e_1'x_i=0,$  $2\leq i\leq s.$
Thus we have:
$$\begin{cases}
  [e_i,x_i]=e_i, &  1\le i\leq k-1,\\
  [e_k,x_j]=e_k, &  1\le j\leq s,\\
  [x_i,e_k]=\nu_{i}e_k, &  1\le i\leq k-1,\\
  [x_1,x_j]=\delta_{1,j}e_k, &  1\le j\leq k-1.
\end{cases}$$

\textbf{Case 2.1.1.}
Let $e_k\in  Ann_r(L),$ then $\nu_{i}=0$.
Making the change  $x_1'=x_1-\delta_{1,1}e_k,$ we may assume that  $\delta_{1,1}=0.$
Then from the Leibniz identity
$$0=[x_1,[x_j,x_1]]-[[x_1,x_j],x_1]-[[x_1,x_1],x_j]=-\delta_{1,j}[e_k,x_1]=-\delta_{1,j}e_k, \quad 2\leq j\leq k-1,$$
we get $\delta_{1,j}=0$ for $2\leq j\leq k-1$ and we obtain the algebra $$L_1(\beta_i), \quad  \text{ with }  \beta_i=1  \text{ for }  1 \leq i \leq s \quad \text{and} \quad \beta_i=0 \text{ for }  s+1 \leq i \leq k-1.$$

\textbf{Case 2.1.2.}
Let $e_k\not\in  Ann_r(L)$. Since the following elements
$$\begin{array}{lc} [x_1,x_i]+[x_i,x_1]=\delta_{1,i}e_k, & [x_1,x_1]=\delta_{1,1}e_k,\\[1mm]
[x_i,e_k]+[e_k,x_i]=(\nu_{i}+1)e_k, &  1\leq i\leq s,\\[1mm]
[x_i,e_k]+[e_k,x_i]=\nu_{i}e_k, &  s+1\leq i\leq k-1,\end{array}$$
belong to the right annihilator, we deduce
$$\delta_{1,i}=0, \quad 1 \leq i \leq k-1,\qquad \nu_{i}=-1, \quad 1\leq i\leq s, \qquad \nu_{i}=0, \quad s+1\leq i\leq k-1.$$

Thus, in this case we obtain the algebra
$$L_2(\beta_i), \quad  \text{ with }  \beta_i=1  \text{ for }  1 \leq i \leq s \quad \text{and} \quad \beta_i=0 \text{ for }  s+1 \leq i \leq k-1.$$

\textbf{Case 2.2.}
Let $s=1$, then we have the following table of multiplication
$$\begin{cases}
  [e_1,x_1]=e_1+\beta_{1,1}e_k, \\
  [e_i,x_i]=e_i, &  2\le i\leq k-1,\\
  [e_1,x_j]= \beta_{1,j}e_k, &  2\le j\leq k-1,\\
  [e_k,x_1]=e_k,\\
  [x_i,e_1]=\gamma_{i,1}e_k, &  1\le i\leq k-1,\\
  [x_i,e_k]=\nu_{i,1}e_1+\nu_{i,k}e_k, &  1\leq i\leq k-1,\\
  [x_1,x_j]=\delta_{1,j}e_k, &  1\le j\leq k-1,
\end{cases}$$
where $\nu_{1,1}\beta_{1,1}=0,$
$\nu_{i,1}\beta_{1,1}=0,$ for $s<i\leq k-1.$

\textbf{Case 2.2.1.}
Let $e_k\in  Ann_r(L)$. Then $\nu_{i,1}=\nu_{i,k}=0.$
Next, $[x_1,e_1]+[e_1,x_1]=e_1+(\beta_{1,1}+\gamma_{1,1})e_k\in  Ann_r(L)$. Since $e_k\in  Ann_r(L)$ we have $e_1\in  Ann_r(L)$. Hence $\gamma_{i,1}=0.$

Performing a basis change $x_1'=x_1-\delta_{1,1}e_k$ we get $[x_1',x_1']=[x_1-\delta_{1,1}e_k, x_1-\delta_{1,1}e_k]=0.$
Thus, we may assume $\delta_{1,1}=0.$ Using the Leibniz identity:
$$0=[x_1,[x_j,x_1]]-[[x_1,x_j],x_1]+[[x_1,x_1],x_j]=\delta_{1,j}e_k,$$
we derive $\delta_{1,j}=0$ for $2\leq j\leq k-1,$ from which we obtain the algebra $L_3(\beta_{i}).$

\textbf{Case 2.2.2.}
Let $e_k\not\in  Ann_r(L)$. Then since the following elements $$[e_1,x_j]+[x_j,e_1]=(\beta_{1,j}+\gamma_{j,1})e_k, \quad [x_1,x_1]=\delta_{1,1}e_k, \quad [x_1,x_j]+[x_j,x_1]=\delta_{1,j}e_k$$
belong to the right annihilator, we deduce
$$\gamma_{j,1}=-\beta_{1,j}\quad 2\leq j\leq k-1, \qquad \delta_{1,j}=0, \quad 1\leq j\leq k-1.$$

From the Leibniz identities
$$0=[x_i,[e_1,x_1]-[[x_i,e_1],x_1]+[[x_i,x_1],e_1]=-\beta_{1,i}e_k+\beta_{1,1}(\nu_{i,1}e_1+\nu_{i,k}e_k)+\beta_{1,i}e_k=\beta_{1,1}(\nu_{i,1}e_1+\nu_{i,k}e_k),$$
$$0=[x_i,[e_1,x_j]]-[[x_i,e_1],x_j]+[[x_i,x_j],e_1]=\beta_{1,j}(\nu_{i,1}e_1+\nu_{i,k}e_k), \quad 2 \leq j \leq k-1,$$
we obtain $$\beta_{1,j}\nu_{i,1}=0, \quad \beta_{1,j}\nu_{i,k}=0, \quad 1\leq i\leq k-1, \  1\leq j\leq k-1.$$

Suppose there exists a $j$ such that $\beta_{1,j}\not=0$. Then $\nu_{i,1}=\nu_{i,k}=0$ for $1\leq i\leq k-1$ which implies that $e_k\in  Ann_r(L)$. This contradicts the assumption that $e_k\not\in  Ann_r(L)$. Therefore, $\beta_{1,j}=0$ for $1\leq j\leq k-1$. Thus we have the multiplication:
$$\begin{cases}
  [e_i,x_i]=e_i, &  1\le i\leq k-1,\\
  [e_k,x_1]=e_k,\\
  [x_1,e_1]=\gamma_{1,1}e_k, \\
  [x_i,e_k]=\nu_{i,1}e_1+\nu_{i,k}e_k, &  1\leq i\leq k-1.
\end{cases}$$
Using the Leibniz identity for $2\leq i\leq k-1$, we get:
$$0=[x_i,[x_i,e_k]]-[[x_i,x_i],e_k]+[[x_i,e_k],x_i]=[x_i,\nu_{i,1}e_1+\nu_{i,k}e_k]+[\nu_{i,1}e_1+\nu_{i,k}e_k,x_i]=\nu_{i,k}(\nu_{i,1}e_1+\nu_{i,k}e_k).$$
This implies that $\nu_{i,k}=0$ for $2\leq i\leq k$.

We also have:
$$0=[x_1,[x_1,e_1]]-[[x_1,x_1],e_1]+[[x_1,e_1],x_1]=\gamma_{1,1}\nu_{1,1}e_1+\gamma_{1,1}(1+\nu_{1,k})e_k,$$
$$0=[x_1,[x_1,e_k]]-[[x_1,x_1],e_k]+[x_1,e_k],x_1]=\nu_{1,1}(\nu_{1,k}+1)e_1+(\nu_{1,1}\gamma_{1,1}+\nu_{1,k}(\nu_{1,k}+1))e_k,$$
$$0=[x_1,[x_i,e_k]]-[[x_1,x_i],e_k]+[[x_1,e_k],x_i]=\nu_{i,1}\gamma_{1,1}e_k,$$
$$0=[x_i,[x_1,e_k]]-[[x_i,x_1],e_k]+[[x_i,e_k],x_1]=\nu_{i,1}(\nu_{i,k}+1)e_1.$$

Thus, $$\begin{array}{lll}\gamma_{1,1}\nu_{1,1}=0, & \gamma_{1,1}(1+\nu_{1,k})=0,\\
\nu_{1,1}(\nu_{1,k}+1)=0, & \nu_{1,k}(\nu_{1,k}+1)=0,\\
\nu_{i,1}\gamma_{1,1}=0, & \nu_{i,1}(\nu_{1,k}+1)=0, & 2\leq i\leq k-1.\end{array}$$

If $\nu_{1,k}=0$, then $\gamma_{1,1}=\nu_{i,1}=0$ for $1\leq i\leq k-1,$ which implies that $e_k\in Ann_r(L).$
This is a contradiction with assumption that $e_k\notin Ann_r(L).$ Thus, $\nu_{1,k}=-1.$

If $\gamma_{1,1}\not=0$, then $\nu_{i,1}=0$ for $1\leq i\leq k-1$. The multiplication is as follows:
$$\begin{cases}
  [e_i,x_i]=e_i, &  1\le i\leq k-1,\\
  [e_k,x_1]=e_k,\\
  [x_1,e_1]=\gamma_{1,1}e_k, \\
  [x_1,e_k]=-e_k.
\end{cases}$$
Making the basis change $e_1'=e_1+\gamma_{1,1}e_k$ we have the algebra
$$L_1(\beta_i), \quad  \text{ with }  \beta_1=1  \quad \text{and} \quad \beta_i=0 \text{ for }  2 \leq i \leq k-1.$$

If $\gamma_{1,1}=0$, then we obtain the algebra $L_4(\nu_i).$

\textbf{Case 3.}
Let $\beta_{k,1}=\beta_{k,2}=\dots=\beta_{k,k-1}=0$.
Then the multiplication is:
$$\begin{cases}
  [e_i,x_i]=e_i+\beta_{i,i}e_k, &  1\le i\leq k-1,\\
  [e_i,x_j]= \beta_{i,j}e_k, &  1\le i,j\leq k-1,\; i\not=j,\\
  [x_i,e_j]=\gamma_{i,j}e_k, &  1\le i,j\leq k-1,\\
  [x_i,e_k]=  \sum\limits_{j=1}^k\nu_{i,j}e_j, &  1\le i\leq k-1,\\
  [x_i,x_j]=\delta_{i,j}e_k, &  1\le i,j\leq k-1.
\end{cases}$$

From the Leibniz identity
$$0=[e_i,[x_j,x_i]]-[[e_i,x_j],x_i]+[[e_i,x_i],x_j]=\beta_{i,j}e_k,$$
we have  $$\beta_{i,j}=0, \quad 1\leq i, j\leq k-1,\quad i\not=j.$$

Then making the change $e_i'=e_i+\beta_{i,i}e_k$ for $1\leq i\leq k-1,$ we get $$[e_i',x_i]=[e_i+\beta_{i,i}e_k, x_i]=e_i'.$$

Thus, we may suppose $\beta_{i,i}=0$ for
$1 \leq i \leq k-1.$

Notice that
$$0=[x_i,[e_k,x_j]]-[[x_i,e_k],x_j]+[[x_i,x_j],e_k]=-\nu_{i,j}e_j$$
implies that $\nu_{i,j}=0$ for $1\leq i,j\leq k-1$. This gives $[x_i,e_k]=\nu_{i,k}e_k$ for $1\leq i\leq k-1$.

Using the Leibniz identity, we have:
$$0=[x_j,[e_i,x_i]]-[[x_j,e_i],x_i]+[[x_j,x_i],e_i]=\gamma_{i,j}e_k,$$
$$0=[x_i,[x_i,e_k]]-[[x_i,x_i],e_k]+[[x_i,e_k],x_i]=\nu_{i,k}^2e_k,$$
which implies that $\gamma_{i,j}=\nu_{i,k}=0$ for $1\leq i,j\leq k-1$.

Our final multiplication is:
$$L_5(\delta_{i,j}):\begin{cases}
  [e_i,x_i]=e_i, &  1\le i\leq k-1,\\
  [x_i,x_j]=\delta_{i,j}e_k, &  1\le i,j\leq k-1.
\end{cases}$$
\end{proof}

Now we give the description of solvable algebras $R(\mathbf{a_k}, k-1)$ in the case of $\alpha_i=-1$ for $1 \leq i \leq k-1.$

\begin{thm} Let $L$ be a solvable Leibniz algebra from the class $R(\mathbf{a_k}, k-1)$ and $\alpha_i=-1$ for $1 \leq i \leq k-1.$ Then $L$ is isomorphic to one of the following algebras:
$$
\begin{array}{ll}L_{6}(\beta_j):\begin{cases}
  [e_i,x_i]=e_i, &  1\le i\leq k-1,\\
  [e_k,x_j]= \beta_{j}e_k, &  1\le j\leq k-1,\\
  [x_i,e_i]=-e_i, &  1\le i\leq k-1,
\end{cases} & L_{7}(\beta_j):\begin{cases}
  [e_i,x_i]=e_i, &  1\le i\leq k-1,\\
  [e_k,x_j]= \beta_{j}e_k, &  1\le j\leq k-1,\\
  [x_i,e_i]=-e_i, &  1\le i\leq k-1,\\
  [x_j,e_k]= -\beta_{j}e_k, &  1\le j\leq k-1,
\end{cases} \\ L_{8}(\gamma_{i}):\begin{cases}
  [e_i,x_i]=e_i, &  1\le i\leq k-1,\\
  [e_k,x_1]=e_k, \\
  [x_i,e_i]=-e_i, &  1\le i\leq k-1,\\
  [x_i,e_1]=\gamma_{i}e_k, &  1\le i\leq k-1,
\end{cases} & L_{9}: \begin{cases}
  [e_1,x_1]=e_1 + \beta_{1}e_k,\\
  [e_i,x_i]=e_i, & 2\le i\leq k-1,\\
  [e_1,x_i]=\beta_ie_k, & 2\le i\leq k-1,\\
  [e_k,x_1]=e_k, & \\
  [x_1,e_1]=-e_1-\beta_{1}e_k,  &\\
  [x_i,e_i]=-e_i,  & 2\le i\leq k-1,\\
  [x_i,e_1]=-\beta_{i}e_k, & 1\le i\leq k-1,\\
  [x_1,e_k]=-e_k.
\end{cases}\\ L_{10}(\delta_{i,j}):\begin{cases}
  [e_i,x_i]=e_i, &  1\le i\leq k-1,\\
  [x_i,e_i]=-e_i, &  1\le i\leq k-1,\\
  [x_i,x_j]=\delta_{i,j}e_k, &  1\le i,j\leq k-1.
\end{cases}\end{array}$$

\end{thm}

\begin{proof}
The proof is similar to the proof of the Theorem \ref{thm3.1}
\end{proof}

Now we give the description of solvable Leibniz algebras from the class $R(\mathbf{a_k}, k-1)$ in the general case.
Let there exist $i$ and $j$ such that $\alpha_i=-1$ and $\alpha_j=0.$ Without loss of generality we can assume that $\alpha_1=\dots =\alpha_{t-1}=-1$ and
$\alpha_t=\dots =\alpha_{k-1}=0.$

\begin{thm}\label{thm3.3} Let $L$ be a solvable Leibniz algebra from the class $R(\mathbf{a_k}, k-1)$ and let  $\alpha_1=\dots =\alpha_{t-1}=-1$ and
$\alpha_t=\dots =\alpha_{k-1}=0.$  Then $L$ is isomorphic to one of the following algebras:
$$M_{1, t}(\beta_{i}):\begin{cases}
  [e_i,x_i]=e_i, &  1\le i\leq k-1,\\
  [e_k,x_i]= \beta_{i}e_k, &  1\le i\leq k-1,\\
  [x_i,e_i]=-e_i, &  1\le i\leq t-1,
\end{cases} \quad M_{2,t}(\beta_{i}):\begin{cases}
  [e_i,x_i]=e_i, &  1\le i\leq k-1,\\
  [e_k,x_i]= \beta_{i}e_k, &  1\le i\leq k-1,\\
  [x_i,e_k]= -\beta_{i}e_k, &  1\le i\leq k-1,\\
  [x_i,e_i]=-e_i, &  1\le i\leq t-1,
\end{cases} $$
$$M_{3,t}(\beta_{i}):\begin{cases}
  [e_{t},x_{t}]=e_{t} + \beta_{t}e_k,\\
  [e_i,x_i]=e_i, &  1\le i\leq k-1, \\
  [e_{t},x_i]=\beta_ie_k, &  1\le i\leq k-1, \\
  [e_k,x_{t}]= e_k,\\
  [x_i,e_i]=-e_i, &  1\le i\leq t-1,
\end{cases} \quad M_{4,t}(\beta_{i}):\begin{cases}
  [e_1,x_1]=e_1 + \beta_{1}e_k,\\
  [e_i,x_i]=e_i, &  2\le i\leq k-1,\\
  [e_1,x_i]=\beta_ie_k, &  2\le i\leq k-1,\\
  [e_k,x_1]= e_k,\\
  [x_1,e_1]=-e_1-\beta_1e_k,& \\
  [x_i,e_i]=-e_i, &  2\le i\leq t-1,\\
  [x_i,e_1]=-\beta_ie_k, &  2\le i\leq k-1,\\
  [x_1,e_k]=-e_k,
\end{cases} $$
$$M_{5,t}(\gamma_{i}):\begin{cases}
    [e_i,x_i]=e_i, &  1\le i\leq k-1,\\
  [e_k,x_1]= e_k,\\
  [x_i,e_i]=-e_i, &  1\le i\leq t-1,\\
    [x_i,e_1]=\gamma_ie_k, &  2\le i\leq k-1,
\end{cases} \quad M_{6,t}(\nu_{i}):\begin{cases}
  [e_i,x_i]=e_i, &  1\le i\leq k-1,\\
  [e_k,x_t]= e_k,\\
  [x_i,e_i]=-e_i, &  1\le i\leq t-1,\\
  [x_t,e_k]=-e_k,\\
  [x_i,e_k]=\nu_ie_t, &  1\le i\leq k-1,
\end{cases} $$
$$M_{7,t}(\delta_{i,j}):\begin{cases}
  [e_i,x_i]=e_i, &  1\le i\leq k-1,\\
  [x_i,e_i]=-e_i, &  1\le i\leq t-1,\\
  [x_i,x_j]=\delta_{i,j}e_k, &  1\le i,j\leq k-1.
\end{cases}$$

\end{thm}

\begin{proof}
The proof is similar to the proof of the Theorem \ref{thm3.1}
\end{proof}

In the following theorem we give the classification of $2k-1$-dimensional solvable Leibniz algebras with $k$-dimensional abelian nilradical.

\begin{thm} Let $L$ be a $2k-1$-dimensional solvable Leibniz algebra with $k$-dimensional abelian nilradical. Then $L$ is isomorphic to one of the following pairwise non-isomorphic algebras:
$$\begin{array}{ll}M_{1,t}(\beta_{1}, \beta_{2},\dots, \beta_{k-1}),
& M_{4,t}(1, \beta_2, \beta_3, \dots, \beta_{t-1}, \beta_t,\beta_{t+1}, \dots, \beta_{k-1}),
 \\ M_{2,t}(\beta_1, \dots, \beta_{t-1}, 0, \dots, 0), & M_{4,t}(0,1, \beta_3, \dots, \beta_{t-1}, \beta_t,\beta_{t+1}, \dots, \beta_{k-1}),\\
M_{3,t}(1, \beta_2,  \dots, \beta_{t-1}, \beta_t,\beta_{t+1}, \beta_{t+2} \dots, \beta_{k-1}), & M_{4,t}(0,0,0, \dots, 0, 1, \beta_{t+1}, \dots, \beta_{k-1}),\\
M_{3,t}(0,0, \dots, 0, 1, \beta_{t+1},\beta_{t+2} \dots, \beta_{k-1}), & M_{5,t}(1, \gamma_3, \dots, \gamma_{t-1}, \beta_t, \dots, \gamma_{k-1}),\\
M_{3,t}(0,0, \dots, 0, 0, 1,\beta_{t+2} \dots, \beta_{k-1}), & M_{5,t}(0,0, \dots, 0, 1, \gamma_{t+1}, \dots, \gamma_{k-1}), \\  M_{7,t}(\delta_{i,j}).
 \end{array}$$
where at least one of the parameters $\delta_{i,j}$ is non-zero and this non-zero parameter can be scaled to $1.$
\end{thm}

\begin{proof} From Theorem \ref{thm3.3} we have the list of solvable Leibniz algebras from the class $R(a_k, k-1)$. It is obvious that the class $M_{1,t}(\beta_{i})$ gives us pairwise non-isomorphic algebras for any parameters $\beta_i \in \mathbb{C}.$ Moreover, in the case $t=1$ we get the algebra $L_1(\beta_i.)$

In the class of algebras $M_{2,t}(\beta_{i}),$ at least one of the parameters $\beta_i$ is non-zero. Otherwise we obtain the algebra $M_{1,t}(0,0,\dots, 0).$
 Moreover if $\beta_j \neq 0$ with $j \geq t$ then making the change
 $$e_t'=e_k, \quad e_j'=e_t, \quad e_k'=e_j,$$
 $$x_t'=\frac 1 {\beta_j}x_j, \quad x_j'=x_t - \frac {\beta_t}{\beta_j}x_j, \quad x_i'=x_i - \frac {\beta_i}{\beta_j}x_j, \quad 1 \leq i \leq k-1,$$
 we obtain that
  $$M_{2,t}(\beta_1, \dots , \beta_{j}, \dots \beta_{k-1}) \cong M_{1,t+1}(-\frac{\beta_1}{\beta_j}, \dots, \frac{1}{\beta_j}, \dots, -\frac{\beta_{k-1}}{\beta_j} ).$$

Thus, we have that $M_{2,t}(\beta_1, \dots, \beta_{t-1}, 0, \dots, 0)$ are pairwise non-isomorphic algebras to the $M_{1,t}$ algebras.

 In the class of algebras $M_{3,t}(\beta_{1}, \beta_2, \dots, \beta_{k-1})$ also at least one of the parameters $\beta_i$ is non-zero.
  Moreover, if $\beta_j\neq 0$ with $1 \leq  j \leq t-1$, then without lost of generality we may suppose $\beta_1 \neq 0$ and making the change $e_k' =\beta_1e_k$ we may assume $\beta_1=1.$ In the case of $\beta_j= 0$ for $1 \leq  j \leq t-1$ and $\beta_t \neq 0,$ the parameter $\beta_t$ can be scaled to $1.$
  If $\beta_j= 0$ for $1 \leq  j \leq t,$ then without lost of generality we may assume $\beta_{t+1} \neq 0$ and making the change  $e_k' =\beta_{t+1}e_k$ we obtain $\beta_{t+1}=1.$

 Thus,
$$\begin{array}{ll}M_{3,t}(1, \beta_2,  \dots, \beta_{t-1}, \beta_t,\beta_{t+1}, \beta_{t+2} \dots, \beta_{k-1})\\
M_{3,t}(0,0, \dots, 0, 1, \beta_{t+1},\beta_{t+2} \dots, \beta_{k-1})\\
M_{3,t}(0,0, \dots, 0, 0, 1,\beta_{t+2} \dots, \beta_{k-1}). \end{array}$$
are non-isomorphic algebras.

Analyzing the class of algebras $M_{4,t}(\beta_{1}, \beta_2,\dots, \beta_{k-1})$ and $M_{5,t}(\gamma_{1}, \gamma_2, \dots, \gamma_{k-1})$ similarly  we obtain following non-isomorphic algebras

$$\begin{array}{ll}M_{4,t}(1, \beta_2, \beta_3, \dots, \beta_{t-1}, \beta_t,\beta_{t+1}, \dots, \beta_{k-1})\\
M_{4,t}(0,1, \beta_3, \dots, \beta_{t-1}, \beta_t,\beta_{t+1}, \dots, \beta_{k-1})\\
M_{4,t}(0,0,0, \dots, 0, 1, \beta_{t+1}, \dots, \beta_{k-1}) \end{array}, \quad \begin{array}{ll}M_{5,t}(1, \gamma_3, \dots, \gamma_{t-1}, \beta_t, \dots, \gamma_{k-1})\\
M_{5,t}(0,0, \dots, 0, 1, \gamma_{t+1}, \dots, \gamma_{k-1}). \end{array}$$

In the class of $M_{6,t}(\nu_{1}, \nu_{2}, \dots, \nu_{k-1})$ making the change
$$e_1'=e_k, \quad e_t'=e_1, \quad e_k'=e_t, \qquad x_1'=x_t, \quad x_t'=x_1,$$
we get that $M_{6,t}(\nu_{1}, \nu_{2}, \dots, \nu_{k-1}) \cong M_{5,t+1}(\nu_{1}, \nu_{2}, \dots, \nu_{k-1})$.

\end{proof}

\textbf{Acknowledgements.}
This material is based upon work supported by the National Science Foundation under Grant No. NSF 1658672.


\begin{thebibliography}{9}

\bibitem{Ad}
{\rm J. Q. Adashev, M. Ladra, B. A. Omirov}
Solvable Leibniz Algebras with Naturally Graded Non-Lie $p$-Filiform Nilradicals.
{\it Communications in Algebra}, 45(10), 2017, 4329--4347.

\bibitem{Ancochea}{\rm Ancochea J.M., Campoamor-Stursberg R., García L.}
{Indecomposable Lie algebras with nontrivial Levi decomposition cannot have filiform radical}, {\it Int. Math. Forum},  1 (5-8), 2006, 309--316.

\bibitem{Bar}
{\rm Barnes D.W.}, On Levi's theorem for Leibniz algebras, {\it Bull.
  Aust. Math. Soc.}, 86(2), 2012, 184--185.

\bibitem{Bl} {\rm Blokh A.},  On a generalization of the concept of Lie algebra, {\it Dokl.
Akad. Nauk SSSR}, 165, 1965, 471--473.

\bibitem{CanKhud}{\rm Ca\~{n}ete E.M.,   Khudoyberdiyev A.Kh.,}
 The classification of 4-dimensional Leibniz algebras, {\it Linear Algebra and its Applications}, 439(1), 2013, 273--288.

\bibitem{Casas}
J. M. Casas, M. Ladra, B. A. Omirov, I. A. Karimjanov
{Classification of solvable Leibniz algebras with null-filiform nilradical}.
{\it Linear and Multilinear Algebra}, 61 (6), 2013, 758--774.

\bibitem{Casas 3-dim}{\rm Casas J.M., Insua M.A., Ladra M., Ladra S.,}  An algorithm for the classification of 3-dimensional complex
Leibniz algebras, {\it Linear Algebra and its Applications}, 436, 2012, 3747--3756.

\bibitem{Jac} {\rm Jacobson N. }{Lie algebras,} {\it Interscience Publishers}, Wiley, New York, 1962.

\bibitem{Khud}
A. Kh. Khudoyberdiyev, I. S. Rakhimov, Sh. K. Said Husain
{On classification of 5-dimensional solvable Leibniz algebras}.
{\it Linear Algebra and its Applications}, 457, 2014, 428--454.

\bibitem{LinBos}  {\rm Lindsey Bosko-Dunbar, Matthew Burke, Jonathan D. Dunbar, J.T. Hird, Kristen Stagg,}
 Solvable Leibniz algebras with abelian nilradical, 2014,  arXiv[Math]:1409.0936. 23 p.

\bibitem{Lad}  {\rm Khudoyberdiyev A.Kh., Ladra M., Omirov B.A.,} On solvable
Leibniz algebras whose nilradical is a direct sum of
null-fililiform algebras, {\it Linear and Multilinear Algebra},
62 (9), 2014, 1220--1239

\bibitem{Loday1} {\rm Loday J.L.,}  Une version non commutative des
alg$\acute{e}$bres de Lie: les alg$\acute{e}$bres de Leibniz,
{\it Ens. Math.}, 39, 1993, 269--293.

\bibitem{Muba}
{\rm  Mubarakzjanov G.M.}, On solvable {L}ie algebras ({R}ussian), {\it Izv.
  Vys\v s. U\v cehn. Zaved. Matematika}, 1963, 114--123.

  \bibitem{NdWi}
{\rm Ndogmo J.C., Winternitz P.},  Solvable {L}ie algebras with abelian
  nilradicals, {\it J. Phys. A},  27, 1994, 405--423.

  \bibitem{SnWi}
{\rm {\v{S}}nobl L. Winternitz P.}, A class of solvable {L}ie
algebras and their {C}asimir invariants, {\it J. Phys. A},  38,
2005, 2687--2700.


\end{thebibliography}
\end{document}